\documentclass[12pt]{article}
\usepackage[intlimits]{amsmath}
\usepackage{amsfonts,amssymb,amscd,amsthm}

\def\ind{\operatorname{ind}}

\newcommand{\im}{\mathop{\rm Im}}

\newcommand{\Vect}{\mathop{\rm Vect}}

\newcommand{\spin}{\mathop{\rm spin}}
\newcommand{\Mat}{\mathop{\rm Mat}}

\newtheorem{theorem}{Theorem}[section]
\newtheorem{lemma}[theorem]{Lemma}
\newtheorem{proposition}[theorem]{Proposition}

\theoremstyle{definition}
\newtheorem{definition}[theorem]{Definition}

\newtheorem{assumption}[theorem]{Assumption}
\newtheorem{example}[theorem]{Example}
\newtheorem{remark}[theorem]{Remark}

\def\im{\operatorname{Im}}

\title{Poincar\'e isomorphism in $K$-theory \\
on manifolds with edges}
\author{V.~E.~Nazaikinskii, A.~Yu.~Savin, and B.~Yu.~Sternin}

\begin{document}

\maketitle
\begin{abstract}
The aim of this paper is to construct the Poincar\'e isomorphism in
$K$-theory on manifolds with edges.  We show that the Poincar\'e
isomorphism can naturally be constructed in the framework of
noncommutative geometry. More precisely, to a manifold with edges
we assign a noncommutative algebra  and construct an isomorphism
between the $K$-group of this algebra and the $K$-homology group of
the manifold with edges viewed as a compact topological space.
\end{abstract}

\section{Introduction}

Let $M$ be a smooth closed even-dimensional manifold equipped with
a complex spin structure  ($\spin^c$-structure in what follows).
Then in $K$-theory we have the Poincar\'e isomorphism
\begin{equation}
\label{smooth1} \tau\colon  K^0(M) \longrightarrow K_0(M)
\end{equation}
between the $K$-group of vector bundles on the manifold and the
$K$-homology group~\cite{BaDo1,BaDo4,Kam1}. From the viewpoint of
analysis (to which we stick here), the $K$-homology group is
naturally identified with the group generated by elliptic operators
on $M$. In this language, the mapping \eqref{smooth1} takes the
class of a vector bundle on $M$ to the class of the $\spin^c$ Dirac
operator twisted by the vector bundle.

The aim of this paper is to construct a Poincar\'e isomorphism
similar to~\eqref{smooth1} for the case of manifolds with edges.
Manifolds with edges are a class of manifolds with nonisolated
singularities. In this situation, we show that the Poincar\'e
isomorphism can naturally be constructed in the framework of
noncommutative geometry. More precisely, to a manifold
$\mathcal{M}$ with edges we assign a noncommutative algebra
$\mathcal{A}$ and obtain an isomorphism
\begin{equation}
\label{new1} K_0(\mathcal{A})\longrightarrow K^0(C(\mathcal{M}))
\end{equation}
between the $K$-group of $\mathcal{A}$ and the analytic
$K$-homology group of the algebra $C(\mathcal{M})$ of continuous
functions on the manifold with edges viewed as a compact
topological space. If $\mathcal{M}$ is a smooth manifold, then
$\mathcal{A}$ is just $C(\mathcal{M})$ and \eqref{new1} is none
other than~\eqref{smooth1} (the $K$-groups of spaces being
identified with the $K$-groups of the corresponding function
algebras).

The main difficulty in constructing the isomorphism~\eqref{new1} is
that one has to find a realization of the Dirac operator on the
smooth part of $\mathcal{M}$ as an elliptic operator over the
algebra  $C(\mathcal{M})$ of continuous functions on $\mathcal{M}$.
It turns out that this is not always possible. The corresponding
obstruction was computed in \cite{NaSaSt3} and is a generalization
of the Atiyah--Bott obstruction~\cite{AtBo2} in the theory of
classical boundary value problems.  For the case in which the
obstruction vanishes, we construct a Fredholm realization of the
Dirac operator in the class of boundary value problems introduced
in \cite{SaSt11}.

The research was supported by the RFBR under grant No.~06-01-00098
and by DFG in the framework of the project DFG 436 RUS
113/849/0-1\circledR ``$K$-theory and noncommutative geometry of
stratified manifolds." This research was carried out during our
stay at the Institute for Analysis, Hannover University (Germany).
We are grateful to Professor E.~Schrohe and other members of the
university staff for their kind hospitality.

\section{Geometry}\label{pardva}

\paragraph{Manifolds with edges.}
Let $M$ be a compact smooth manifold with boundary $\partial M$,
and suppose that $\partial M$ is the total space of a smooth
locally trivial bundle
$$
\pi \colon \partial M\longrightarrow X
$$
over a closed smooth base $X$ with fiber
$\pi^{-1}(x)\equiv\Omega_x$ over $x\in X$, which is a smooth closed
manifold. For simplicity, we assume that $X$ and $\Omega_x$ are
connected.

\begin{definition}
The topological space obtained from  $M$ by identifying points in
each fiber of $\pi$ is called the \emph{manifold with edge}
corresponding to the pair $(M,\pi)$. Let $\mathcal{M}$ denote the
manifold with edge.
\end{definition}

In this paper, we construct Poincar\'e duality for the case in
which $M$ and $X$ are even-dimensional manifolds (and hence the
fiber $\Omega$ is odd-dimensional automatically).

\paragraph{1. Riemannian metric.}

We fix a Riemannian metric on $M$ and assume that it is of product
type
$$
g_M=dt^2+g_{\partial M}
$$
in a collar neighborhood $\partial M\times [0,1)$ of the boundary,
where $t\in [0,1)$ is the coordinate normal to the boundary and
$g_{\partial M}$ is a metric on the boundary compatible with $\pi$
(a so-called \emph{submersion metric}). Namely, a connection in
$\pi$ gives the direct sum decomposition
\begin{equation}\label{vertik1}
T\partial M\simeq T^H\partial M\oplus T^V\partial M\simeq
\pi^*TX\oplus T\Omega
\end{equation}
of the tangent bundle into horizontal and vertical components. The
vertical component coincides with the tangent bundle to the fibers,
and the horizontal component is naturally isomorphic to $\pi^*TX$.
Using the decomposition~\eqref{vertik1}, one defines a submersion
metric on the boundary as
\begin{equation}
\label{vertical} g_{\partial M}=\pi^*g_X+g_\Omega,
\end{equation}
where $g_X$ is a metric on $X$ and $g_\Omega$ is a family of
metrics on the fibers $\Omega$.

For the cotangent bundle $T^*\partial M$, there is a similar
decomposition
\begin{equation}
\label{vertik2} T^*\partial M\simeq \pi^*T^*X\oplus T^*\Omega.
\end{equation}

{\small Recall how the decompositions \eqref{vertik1} and
\eqref{vertik2} are defined. Let
$$
\omega\colon T\partial M\to T\Omega,\qquad \omega|_{T\Omega}=Id,
$$
be a connection in $\pi$ (a projection onto the vertical
subbundle). Then \eqref{vertik1} is, by definition,
$$
T\partial M\simeq \ker\omega\oplus T\Omega.
$$
The isomorphism $\ker\omega\stackrel\simeq\to \pi^*TX$ is induced
by the projection $\pi_*\colon T\partial M\to TX$.

In a similar way,  \eqref{vertik2} is defined as
$$
T^*\partial M\simeq \pi^*T^*X\oplus (\ker \omega)^\perp,
$$
where the fiber of the subbundle $(\ker \omega)^\perp$ is, by
definition, the space of $1$-forms annihilating the horizontal
subspace $\ker\omega$. The isomorphism $(\ker\omega)^\perp\to
T^*\Omega$ is induced by restriction of forms to the vertical
subbundle $T\Omega\subset T\partial M$. }

\paragraph{2. $\spin^c$-structure and the Dirac operator.}

Suppose that both the base $X$ and the fibers $\Omega$ are equipped
with $\spin^c$-structures. These data induce a $\spin^c$-structure
on $\partial M$.\footnote{For our purposes, one can use the
following definition (cf.\ \cite{HiRo1}): a $\spin^c$-structure on
$M$ is a bundle $S_M$ of irreducible modules over the bundle
$Cl(M)$ of Clifford algebras on $M$. Here we assume that some
Riemannian metric on $M$ is given. If $M$ is even-dimensional and
oriented, then $S_M\simeq S_M^+\oplus S_M^-$.} We shall also assume
that $M$ is equipped with a $\spin^c$-structure compatible with
that on $\partial M$.

Then we can define the $\spin^c$ Dirac operator
$$
D_M\colon  C^\infty(M, S^+_M)\longrightarrow
C^\infty(M,S^-_M),\qquad S_M\simeq S^+_M\oplus S^-_M,
$$
on $M$. This operator acts on sections of spinor bundles
$S^\pm_M\in \Vect(M)$ (e.g., see~\cite{Gil1}).

\paragraph{3. The principal symbol of the Dirac operator.}

In this paper, we shall only deal with the (principal) symbol of
the Dirac operator. Let us describe it explicitly. First,  by our
choice of the direct product metric near the boundary,
\begin{equation}
\label{dtplusa} D_M\Bigr|_{\partial M\times[0,1)}\simeq \frac
\partial {\partial t}+D_{\partial M},
\end{equation}
where $D_{\partial M}\colon C^\infty(\partial M, S_{\partial
M})\longrightarrow C^\infty(\partial M, S_{\partial M})$ is the
Dirac operator on the boundary.

Second, the spinor bundle of the boundary has the decomposition
(see \cite{BiCh1})
\begin{equation}
\label{fiberspinor} S_{\partial M}\simeq \pi^* S_X\otimes S_\Omega=
\pi^* (S^+_X\oplus S^-_X)\otimes S_\Omega,
\end{equation}
where $S_X\simeq S^+_X\oplus S^-_X\in\Vect(X)$ is the spinor bundle
of the base and $S_\Omega\in\Vect(\partial M)$ is the spinor bundle
of the fibers.

In this notation, the symbol of $D_{\partial M}$ is given by
\begin{equation}\label{dirac1}
\sigma(D_{\partial M} )(\xi,\eta)=
\left(%
\begin{array}{cc}
  1\otimes c_\Omega(\eta) & c_X(\xi)\otimes 1 \\
  c_X(\xi)\otimes 1 & -1\otimes c_\Omega(\eta) \\
\end{array}%
\right)\colon  \begin{array}{c}
  \pi^* S^+_X\otimes S_\Omega \\
  \oplus \\
  \pi^* S^-_X\otimes S_\Omega \\
\end{array}
\longrightarrow \begin{array}{c}
  \pi^* S^+_X\otimes S_\Omega \\
  \oplus \\
  \pi^* S^-_X\otimes S_\Omega \\
\end{array}
,
\end{equation}
where for $T^*\partial M$ we use the decomposition \eqref{vertik2}
and write
$$
(\xi,\eta)\in \pi^*T^*X\oplus T^*\Omega\simeq T^*\partial M
$$
and
\begin{gather*}
c_X\colon  T^*X\otimes S_X\longrightarrow S_X,\quad c_\Omega\colon
T^*\Omega\otimes S_\Omega\longrightarrow S_\Omega,
 \\
c_X(\xi)c_X(\xi')+c_X(\xi')c_X(\xi)=2(\xi,\xi')_{g_X},
 \\
c_\Omega(\eta)c_\Omega(\eta')+c_\Omega(\eta')c_\Omega(\eta)
=2(\eta,\eta')_{g_\Omega},
\end{gather*}
are Clifford multiplications.

Let $Q_0$ be the positive spectral projection of $D_{\partial M}$.
The operator $Q_0$ is a pseudodifferential operator of order zero
on $\partial M$. Its symbol is equal to the positive spectral
projection of $\sigma(D_{\partial M})$ and hence is given by
\begin{equation}
\label{calderon1} \sigma(Q_0)(\xi,\eta)=\frac 1 2
\left(%
\begin{array}{cc}
  1\otimes (1+ c_\Omega(\eta)) & c_X(\xi)\otimes 1 \\
  c_X(\xi)\otimes 1 & 1\otimes (1- c_\Omega(\eta))  \\
\end{array}%
\right)\quad \text{for } |\xi|^2+|\eta|^2=1.
\end{equation}
(This follows from~\eqref{dirac1} in view of the fact that the
square of the symbol~\eqref{dirac1} is a scalar symbol.)

\section{Boundary value problems for twisted Dirac operators}

The fibers of $\pi\colon \partial M\to X$ are odd-dimensional
$\spin^c$-manifolds. Hence we can consider the index (e.g., see
\cite{APS3})
$$
\ind D_\Omega\in K^1(X)
$$
of the family, parametrized by $X$, of self-adjoint Dirac operators
in the fibers. We shall assume that the following condition is
satisfied.
\begin{assumption}
\label{triva}
$$
\ind D_\Omega =0.
$$
\end{assumption}
It will be shown below (see Proposition~\ref{obst1}) that this
condition is necessary for the Dirac operator to define a class in
the $K$-homology of the manifold with edges.

If Assumption~\ref{triva} is satisfied, then one can make the
family $D_\Omega$ invertible by perturbing it by a smooth family of
finite rank operators \cite[Prop.~1]{MePi1}. Doing so if necessary,
we assume from now on that the family $D_\Omega$ is invertible.
Then the family of positive spectral projections
$$
\Pi_+=\Pi_+(D_\Omega)\colon
C^\infty(\Omega,S_\Omega)\longrightarrow C^\infty(\Omega,S_\Omega)
$$
of the Dirac operators in the fibers is smooth. Let $\Pi_-=1-\Pi_+$
be the family of complementary projections.

\paragraph{1. The algebra dual to the algebra of functions on
a manifold with edges.}

Let $\mathcal{T}(\Omega)$ denote the bundle of algebras over $X$
whose fiber
$$
\mathcal{T}(\Omega_x)\subset \mathcal{B}L^2(\Omega_x,S_{\Omega_x})
$$
at the point $x$ is the algebra of zero-order Toeplitz operators on
the range of the projection $\Pi_+(D_{\Omega_x})$. Namely, the
elements of $\mathcal{T}(\Omega_x)$ are operators of the form
$$
F=\Pi_+(D_{\Omega_x}) \widetilde F \Pi_+(D_{\Omega_x})
$$
acting on sections of $S_{\Omega_x}\in \Vect(\Omega_x)$, where
$\widetilde F\colon C^\infty(\Omega_x,S_{\Omega_x})\to
C^\infty(\Omega_x,S_{\Omega_x})$ is a pseudodifferential operator
with scalar symbol $\sigma(\widetilde F)\in C^\infty(\Omega_x)$.
(Note that $\mathcal{T}(\Omega_x)$ is a subalgebra of the algebra
of classical pseudodifferential operators acting on sections of
$S_{\Omega_x}\in \Vect(\Omega_x)$.) By definition, we set
$\sigma(F)=\sigma(\widetilde F)$ (the symbol of the Toeplitz
operator $F$); this is, of course, different from the symbol of $F$
as of a pseudodifferential operator.

Consider the noncommutative algebra
\begin{equation}
\label{algebra1} \mathcal{A}=\{(f,F)\in C^\infty(M)\oplus
C^\infty(X,\mathcal{T}(\Omega))\;|\; \sigma(F)=f|_{\partial
M}\}.
\end{equation}
Its elements are pairs $(f,F)$, where $f$ is a function on $M$, $F$
is a section of $\mathcal{T}(\Omega)$, and the restriction of $f$
to $\partial M$ is equal to $\sigma(F)$.

This algebra naturally acts on the space
$$
C^\infty(M)\oplus C^\infty(\partial M,S_\Omega).
$$
(The action on the subspace $0\oplus
\operatorname{Im}\Pi_-(D_\Omega)$ is defined to be zero.)

Consider the norm closure $\overline{\mathcal{A}}$ of the algebra
$\mathcal{A}$ in the space of operators on $L^2(M)\oplus
L^2(X,L^2(\Omega,S_\Omega))$. (Note that the operators in
$\mathcal{A}$ take the subspace $0\oplus
L^2(X,\Pi_-L^2(\Omega,S_\Omega))$ to zero.) The algebra
$\mathcal{A}$ is a local unital $C^*$-subalgebra of
$\overline{\mathcal{A}}$. Its even $K$-group is generated by formal
differences of projections in matrix algebras over  $\mathcal{A}$.

To a projection\footnote{A projection $\mathcal{P}$ is a compatible
pair $(p,P)$ of projections. The first component defines a
subbundle
$$
\im p \subset M\times \mathbb{C}^n,
$$
and the second component defines a bundle
$$
\im P\subset C^\infty(\partial M,S_\Omega)\otimes \mathbb{C}^n
$$
(in general, infinite-dimensional) over $X$.}
$$
\mathcal{P}=(p,P)\in \Mat(n,\mathcal{A}),
$$
we wish to assign a Fredholm boundary value problem and an element
in the $K$-homology of the algebra $C(\mathcal{M})$ of continuous
functions on $\mathcal{M}$. Here we treat $C(\mathcal{M})$ as the
closure of the algebra $C^\infty(\mathcal{M})\subset C^\infty(M)$
of smooth functions constant on the fibers of $\pi$.

\paragraph{2.Boundary value problems for Dirac operators.}
Consider the Dirac operator on $M$ twisted by the range of $p$
viewed as a vector bundle on $M$:
$$
D_M\otimes 1_{\im p}= (1\otimes p)(D_M\otimes
1_{\mathbb{C}^n})(1\otimes p)\colon  pC^\infty(M,S^+_M\otimes
\mathbb{C}^n)\longrightarrow pC^\infty(M,S^-_M\otimes
\mathbb{C}^n).
$$
It turns out that a compatible pair $(p,P)$ permits one to
construct a well-posed boundary value problem for the twisted Dirac
operator.

By $D_M\otimes 1_{\im\mathcal{P}}$ we denote the boundary value
problem
\begin{equation}
\label{mainbvp} \left\{
\begin{array}{ll}
  (D_M\otimes 1_{\im p})u&=f, \vspace{2mm}\\
  (\Delta^{s/2-1/4}_{\partial M}u_1+v_1)+\widetilde{D}_X^*v_2&=g_1, \vspace{2mm}\\
  (1\otimes (1-P))\left[\widetilde{D}_Xv_1-(\Delta^{s/2-1/4}_{\partial M}u_2+v_2)
  \right]&=g_2, \\
\end{array}%
\right.
\end{equation}
where $s\ge 1$ is some number and
$$
u\in pC^\infty(M,S^+_M\otimes \mathbb{C}^n),\quad f\in
pC^\infty(M,S^-_M\otimes\mathbb{C}^n),
$$
$$
u|_{\partial M}=u_1+u_2\in pC^\infty(\partial M, \pi^* S^+_X\otimes
S_\Omega\otimes\mathbb{C}^n)\oplus pC^\infty(\partial M, \pi^*
S^-_X\otimes S_\Omega\otimes\mathbb{C}^n),
$$
$$
v_1\in (1-P)C^\infty(\partial M,\pi^*S^+_X\otimes S_\Omega\otimes
\mathbb{C}^n),\qquad v_2\in (1-p)C^\infty(\partial
M,\pi^*S^-_X\otimes S_\Omega\otimes \mathbb{C}^n ),
$$
$$
g_1\in C^\infty(\partial M,\pi^*S^+_X\otimes S_\Omega\otimes
\mathbb{C}^n),\quad g_2\in (1-P)C^\infty(\partial
M,\pi^*S^-_X\otimes S_\Omega\otimes \mathbb{C}^n ).
$$
Here by $\Delta_{\partial M}$ we denote Laplace type operators on
$\partial M$ acting on sections of the corresponding bundles. Let
$$
\widetilde{D}_X,\widetilde{D}^*_X\colon  C^\infty(\partial
M,\pi^*S^\pm_X\otimes S_\Omega\otimes \mathbb{C}^n)\longrightarrow
C^\infty(\partial M,\pi^*S^\mp_X\otimes S_\Omega\otimes
\mathbb{C}^n)
$$
be pseudodifferential operators of order zero on $\partial M$ with
principal symbol
\begin{equation}
\label{volna1} c_X(\xi)\otimes 1_{S_\Omega\otimes
\mathbb{C}^n},\quad |\xi|^2+|\eta|^2=1.
\end{equation}
(This symbol is smooth on $T^*\partial M$ away from the zero
section, since $c_X(\xi)$ is a linear function.)

\begin{example}
If the edge $X$ is a point, then we can take
$\mathcal{P}=(1,\Pi_+)$. (Assumption~\ref{triva} is satisfied
automatically, since $K^1(pt)=0$.) In this case, the last equation
in \eqref{mainbvp} disappears, and we obtain the boundary value
problem
$$
\left\{
\begin{array}{ll}
   D_M u&=f, \vspace{2mm}\\
  \Delta^{s/2-1/4}_{\partial M}u_1+v_1&=g_1, \\
\end{array}%
\right.
$$
where $v_1\in (1-\Pi_+)C^\infty(\partial M, S_\Omega), g_1\in
C^\infty(\partial M, S_\Omega)$. The index of this boundary value
problem is equal to the index of the Atiyah--Patodi--Singer
boundary value problem~\cite{APS1}
$$
\left\{
\begin{array}{ll}
   D_M u&=f, \vspace{2mm}\\
  \Pi_+u|_{\partial M}&=g\in \Pi_+C^\infty(\partial M, S_\Omega).
\end{array}%
\right.
$$
\end{example}

The unknown $v_1$ and the right-hand side $g_2$ in \eqref{mainbvp}
belong to subspaces defined as the ranges of families of
pseudodifferential projections in the fibers of $\pi$. In the
general case, the range of such a family is not the space of
sections of any vector bundle on $\partial M$. Therefore, our
boundary value problem is not classical. However, boundary value
problems of this form were studied in~\cite{SaSt11}. Let us use the
finiteness theorem in the cited paper to prove that
problem~\eqref{mainbvp} defines a Fredholm operator in appropriate Sobolev
spaces.

\begin{proposition}\label{lem1}
Let $s\ge 1$. Then the boundary value problem \eqref{mainbvp}
defines a Fredholm operator
\begin{equation}\label{trio1}
D_M\otimes 1_{\im \mathcal{P}}\colon
\begin{array}{c}
  pH^s(M,S^+_M\otimes\mathbb{C}^n) \\
  \oplus \\
  (1-P)L^2(\partial M,\pi^*S^+_X\otimes S_\Omega\otimes \mathbb{C}^n) \\
  \oplus \\
  (1-p)L^2(\partial M,\pi^*S^-_X\otimes S_\Omega\otimes \mathbb{C}^n ) \\
\end{array}
\longrightarrow
\begin{array}{c}
  pH^{s-1}(M,S^-_M\otimes\mathbb{C}^n) \\
  \oplus \\
  L^2(\partial M,\pi^*S^+_X\otimes S_\Omega\otimes \mathbb{C}^n) \\
  \oplus \\
  (1-P)L^2(\partial M,\pi^*S^-_X\otimes S_\Omega\otimes \mathbb{C}^n ) \\
\end{array}.
\end{equation}
\end{proposition}

\begin{proof}

Making the standard reduction of the boundary value problem to the
boundary, we see that it suffices to prove the Fredholm property of
the system of equations
\begin{equation}\label{sist1}
\left\{
\begin{array}{ll}
  (\Delta^{s/2-1/4}_{\partial M}u_1+v_1)+\widetilde{D}_X^*v_2&=g_1, \vspace{2mm}\\
  (1\otimes (1-P))\left[\widetilde{D}_Xv_1-(\Delta^{s/2-1/4}_{\partial M}u_2+v_2)
  \right]&=g_2, \\
\end{array}%
\right.
\end{equation}
on the boundary, where the pair $(u_1,u_2)$ is in the range of the
Calder\'on projection of the Dirac operator $D_M\otimes 1_{\im p}$.
This Calder\'on projection is modulo compact operators equal to the
positive spectral projection of the Dirac operator $D_{\partial
M}\otimes 1_{\im p}$ on the boundary. Hence we can use the spectral
projection, which we denote by $Q$, instead of the Calder\'on
projection in the proof of the Fredholm property of~\eqref{sist1}.
As we mentioned earlier, $Q$ is a pseudodifferential operator with
symbol $\sigma(Q)=\sigma(Q_0)\otimes 1_{\im p}$, where the symbol
of $Q_0$ is given in~\eqref{calderon1}.

The operators on $\partial M$ occurring in~\eqref{sist1} are
special cases of operators with discontinuous symbols on fibered
manifolds~\cite{SaSt11}. The Fredholm property of such operators is
equivalent to the invertibility of the \emph{principal symbol}
$\sigma_{\partial M}$ on $T^*\partial M\setminus \pi^*T^*X$ and the
\emph{operator-valued symbol} $\sigma_X$ defined on $T^*X\setminus
0$ and ranging in pseudodifferential operators in the fibers.

The invertibility of both symbols in our case can be proved by a
straightforward computation.

\paragraph{A. Invertibility of the operator-valued symbol
$\sigma_X$.}  We wish to show that the system
\begin{equation}\label{sistema1}
 \left\{\begin{array}{c}
   u_1+v_1+(c_X(\xi)\otimes 1)v_2=g_1 \vspace{2mm}\\
   (c_X(\xi)\otimes 1)v_1-[1\otimes (1-P)](u_2+v_2)=g_2 \\
 \end{array}\right.,
\qquad |\xi|=1,
\end{equation}
is uniquely solvable for all $x\in X$ and $\xi\in S^*_xX$, where
$$
u=(u_1,u_2)\in\im\sigma_X(Q)(\xi),\quad \sigma_X(Q)(\xi)=\frac 1 2 \left(%
\begin{array}{cc}
  1\otimes 1 & c_X(\xi)\otimes 1 \\
  c_X(\xi)\otimes 1 & 1\otimes 1 \\
\end{array}%
\right),
$$
$$
v_1\in (S^+_X\otimes \im(1-P))_x,\quad v_2\in (S^-_X\otimes
L^2(\Omega, S_\Omega\otimes \im(1-p)))_x,
$$
$$
g_1\in (S^+_X\otimes L^2(\Omega,S_\Omega\otimes
\mathbb{C}^n))_x,\quad g_2\in (S^-_X\otimes \im (1-P))_x.
$$
Note that \eqref{sistema1} is obtained from \eqref{sist1} if we
replace operators by their operator-valued symbols. We have also
used the fact that the operator-valued symbol of a family of
pseudodifferential operators in the fibers is equal to the family
itself (we apply this property to $P$) and the operator-valued
symbol of an operator with smooth symbol is equal to the
restriction of the symbol to  $\pi^*T^*X\subset T^*\partial M$ (we
apply this property to $\widetilde{D}_X,\widetilde{D}_X^*$, and
$Q$). Finally, note that we have used the relation $[1\otimes
P,c_X(\xi)\otimes 1]=0$ to obtain \eqref{sistema1} from
\eqref{sist1}.

Let us prove the triviality of the kernel. If
$\sigma_X(Q)(\xi)u=u$, then $u_1=(c_X(\xi)\otimes 1)u_2$, where
$u_2$ is arbitrary. The first equation in \eqref{sistema1} gives
$u_1+(c_X(\xi)\otimes 1 )v_2=-v_1$. By substituting this into the
second equation, we obtain
$$
v_1-(-v_1)=0.
$$
Thus, $v_1=0$ and $u_1+(c_X(\xi)\otimes 1)v_2=0$. Therefore,
$u_1=0$ and $v_2=0$. This shows that the kernel is trivial.

Let us prove the triviality of the cokernel. First,  $g_2$
in~\eqref{sistema1} can be assumed to be zero (by an appropriate
choice of $v_1$). Hence we have to find the solution of the system
$$
\left\{
\begin{array}{c}
  u_1+v_1+(c_X(\xi)\otimes 1)v_2=g\vspace{2mm}\\
  v_1=[1\otimes (1-P)]\Bigl((c_X(\xi)\otimes 1)u_2+(c_X(\xi)\otimes 1)v_2\Bigr).
\end{array}%
\right.
$$
By substituting the second equation into the first equation, we
obtain
$$
u_1+[1\otimes (1-P)](u_1+(c_X(\xi)\otimes 1)v_2)+(c_X(\xi)\otimes
1)v_2=g,
$$
or
$$
[1\otimes (2-P)]\Bigl(u_1+(c_X(\xi)\otimes 1 )v_2\Bigr)=g,
$$
but the operator $2-P$ is invertible (recall that $P$ is a
projection), and $u_1+(c_X(\xi)\otimes 1)v_2$ can be arbitrary.
Thus, the last equation is solvable for any $g$. This shows that
the cokernel is trivial.

\paragraph{B. Invertibility of the principal symbol
$\sigma_{\partial M}$.} Since the dimensions of the bundles where
the symbol acts are equal, it suffices to show that the
corresponding homogeneous equation
\begin{equation}
\label{sistema2} \left\{\begin{array}{c}
  u_1+v_1+(c_X(\xi)\otimes 1)v_2=0,\\
  (c_X(\xi)\otimes 1) v_1-
  \Bigl[1\otimes (1-\sigma_{\partial M}(P)(\xi,\eta))\Bigr](u_2+v_2)=0,\\
\end{array}\right.
\end{equation}
where
$$
u=(u_1,u_2)\in\im\sigma_{\partial M}(Q)(\xi,\eta),$$
$$
\sigma_{\partial M}(Q)(\xi,\eta)=\frac 1 2 \left(%
\begin{array}{cc}
  1\otimes (1+c_\Omega(\eta)) & c_X(\xi)\otimes 1 \vspace{1mm}\\
  c_X(\xi)\otimes 1 & 1\otimes (1- c_\Omega(\eta)) \\
\end{array}%
\right),
$$
$$
v_1\in (\pi^*S^+_X\otimes \im(1-\sigma_{\partial M}(P)))_m,\quad
v_2\in (\pi^*S^-_X\otimes S_\Omega\otimes \im(1-p))_m,
$$
$$
g_1\in (S^+_X\otimes S_\Omega\otimes \mathbb{C}^n))_m,\quad g_2\in
(S^+_X\otimes \im (1-\sigma_{\partial M}(P)))_m,
$$
has only the trivial solution if $m\in \partial M,$ $(\xi,\eta)\in
T^*_m\partial M$, $|\xi|^2+|\eta|^2=1$, and $\eta\ne 0$.

Indeed, since $\sigma_{\partial M}(Q)(\xi,\eta) u=u$, we obtain
$$
(1\otimes (1+ c_\Omega(\eta)))u_1+(c_X(\xi)\otimes 1)u_2=2u_1;
$$
i.e., $(c_X(\xi)\otimes 1)u_2= (1\otimes (1-c_\Omega(\eta)))u_1$.

A computation shows that the principal symbol of the fiberwise
operator $P$ is
$$
\sigma_{\partial M}(P)(\xi,\eta)=p\sigma(\Pi_+)(\eta).
$$
(See the compatibility condition in the definition of $\mathcal{A}$
in \eqref{algebra1}.)

Let us consider two possibilities.

\textbf{1. $\xi\ne 0$.} By substituting the expression for
$\sigma_{\partial M}(P)$ into \eqref{sistema2}, we obtain
\begin{equation}\label{sistema3}
    \left\{\begin{array}{c}
      u_1+v_1+(c_X(\xi)\otimes 1)v_2=0, \vspace{2mm}\\
      |\xi|^2v_1-\Bigl[1\otimes (1-p\sigma(\Pi_+)(\eta))\Bigr]
       \Bigl((1\otimes (1-c_\Omega(\eta)))u_1+(c_X(\xi)\otimes 1)v_2\Bigr)=0. \\
    \end{array}\right.
\end{equation}
The second equation of the system is equivalent to
$$
v_1(1+|\xi|^2)+\Bigl[1\otimes(1-\sigma(\Pi_+)(\eta))\Bigr](1\otimes
c_\Omega(\eta)) u_1=0.
$$
It is now straightforward to prove the desired triviality of the
solution by using the last equation together with the first
equation in \eqref{sistema3}. To this end, we decompose all vectors
in the system  along the ranges of the orthogonal projections
$$
1\otimes p\sigma(\Pi_+)(\eta),\quad 1\otimes
p(1-\sigma(\Pi_+)(\eta)) ,\quad 1\otimes  (1-p)
$$
and match the corresponding components.

\textbf{2. $\xi=0$.} In this case, \eqref{sistema2} is reduced to
\begin{equation}\label{sistema4}
u_1+v_1=0,\quad\Bigl[1\otimes
(1-p\sigma(\Pi_+)(\eta))\Bigr](u_2+v_2)=0.
\end{equation}
In addition, the equation $\sigma_{\partial M}(Q)(0,\eta)u=u$
implies that $u_1\in p\im\sigma(\Pi_+)(\eta)$ and $u_2\in p\ker
\sigma(\Pi_+)(\eta)$. Using this fact and \eqref{sistema4}, it is
straightforward to show that $u_1=0$, $v_1=0$, $u_2=0$, $v_2=0$.

The proof of Proposition~\ref{lem1} is complete.
\end{proof}

\begin{lemma}\label{lem2}
If  $\mathcal{P}=(0,P)$ is  a projection, then the operators
$$
D_M\otimes 1_{\im\mathcal{P}} \qquad \text{and}\qquad D^*_X\otimes
1_{\im P}
$$
are stably homotopic. In particular,
$$
\ind D_M\otimes 1_{\mathcal{P}}=-\ind (D_X\otimes 1_{\im P}).
$$
\end{lemma}
\begin{remark}
In this lemma, the range of the projection $P$ is the space of
sections of a finite-dimensional vector bundle over $X$. Thus, we
can use this bundle to twist the Dirac operator on $X$, and the
index is well defined.
\end{remark}

\begin{proof}
Since the first component of $\mathcal{P}$ is zero, we see that the
boundary value problem is reduced to the operator
\begin{multline*}
\left(%
\begin{array}{cc}
  1 & {{\widetilde{D}}_X^*} \\
  (1\otimes (1-P)) {\widetilde{D}_X} & -1\otimes (1-P) \\
\end{array}%
\right)
\colon \\
\begin{array}{c}
  C^\infty(X,S^+_X\otimes\im (1-P)) \\
  \oplus \\
  C^\infty(X,S^-_X\otimes C^\infty(\Omega, S_\Omega\otimes \mathbb{C}^n)) \\
\end{array}
\longrightarrow
\begin{array}{c}
  C^\infty(X,S^+_X\otimes C^\infty(\Omega, S_\Omega\otimes \mathbb{C}^n)) \\
  \oplus \\
  C^\infty(X,S^-_X\otimes\im (1-P)) \\
\end{array}
\end{multline*}
on the boundary. Modulo compact operators, this operator is just
the direct sum of
\begin{multline}
\left(%
\begin{array}{cc}
  1 &  [1\otimes (1-P)]{\widetilde{D}_X^*} \\
  (1\otimes (1-P)) {\widetilde{D}_X} & -1 \\
\end{array}%
\right)\colon \\ \label{fst1}
\begin{array}{c}
  C^\infty(X,S^+_X\otimes\im (1-P)) \\
  \oplus \\
  C^\infty(X,S^-_X\otimes\im (1-P)) \\
\end{array}
\longrightarrow
\begin{array}{c}
  C^\infty(X,S^+_X\otimes\im (1-P)) \\
  \oplus \\
  C^\infty(X,S^-_X\otimes\im (1-P)) \\
\end{array}
\end{multline}
and
\begin{equation}\label{snd}
\left(%
\begin{array}{cc}
  0 & (1\otimes P){\widetilde{D}_X^*} \\
  0 & 0 \\
\end{array}%
\right)\colon
\begin{array}{c}
  0 \\
  \oplus \\
  C^\infty(X,S^-_X\otimes \im P) \\
\end{array}
\longrightarrow
\begin{array}{c}
  C^\infty(X,S^+_X\otimes \im P) \\
  \oplus \\
  0. \\
\end{array}
\end{equation}
Here we have used the compactness of the commutators $[1\otimes
P,\widetilde{D}_X]$ and $[1\otimes P,\widetilde{D}^*_X]$. (The
commutators of the corresponding principal and operator-valued
symbols in the sense of \cite{SaSt11} are zero.)

The operator \eqref{fst1} is homotopic to $1\oplus(-1)$ in the
class of elliptic operators. (To define the homotopy, we multiply
$\widetilde{D}^*_X$ and $\widetilde{D}_X$ by a parameter
$\varepsilon$ varying from $1$ down to $0$.) The principal symbol
of \eqref{snd} coincides on the cosphere bundle with the principal
symbol of the Dirac operator $D^*_X$ on $X$ twisted by the
finite-dimensional vector bundle $\im P$.

The proof of Lemma~\ref{lem2} is complete.
\end{proof}

\paragraph{3. Classes in analytic $K$-homology of manifolds with edges.}

We now assign a class in the analytic $K$-homology of the algebra
$C(\mathcal{M})$ of continuous functions on $\mathcal{M}$ to the
boundary value problem \eqref{trio1}.

Suppose  that the Sobolev exponent  $s$ is an integer $\ge 1$.
First, we reduce \eqref{trio1} to a boundary value problem in $L^2$
spaces. To this end, consider  the composition of \eqref{trio1}
with appropriate powers of order reduction operators
\begin{equation}
\label{dplus}
 T\colon H^s(M,E)\longrightarrow H^{s-1}(M,E),\qquad s\ge 1,
 \quad E\in\Vect(M).
\end{equation}
(An explicit construction of an elliptic operator \eqref{dplus}
that has index zero and does not require boundary conditions is
given, e.g., in \cite[Prop.~20.3.1]{Hor3}.)

This composition is given by the Fredholm operator
$$
\widetilde{D_M\otimes 1_{\im \mathcal{P}}}\colon
\begin{array}{c}
  pL^2(M,S^+_M\otimes\mathbb{C}^n) \\
  \oplus \\
  (1-P)L^2(\partial M,\pi^*S^+_X\otimes S_\Omega\otimes \mathbb{C}^n) \\
  \oplus \\
  (1-p)L^2(\partial M,\pi^*S^-_X\otimes S_\Omega\otimes \mathbb{C}^n ) \\
\end{array}
\longrightarrow
\begin{array}{c}
  pL^2(M,S^-_M\otimes\mathbb{C}^n) \\
  \oplus \\
  L^2(\partial M,\pi^*S^+_X\otimes S_\Omega\otimes \mathbb{C}^n) \\
  \oplus \\
  (1-P)L^2(\partial M,\pi^*S^-_X\otimes S_\Omega\otimes \mathbb{C}^n ) \\
\end{array}
$$
where
$$
\widetilde{D_M\otimes 1_{\im \mathcal{P}}}=T^{s-1} (D_M\otimes
1_{\im \mathcal{P}}) T^{-s}.
$$
This bounded Fredholm operator acts in subspaces of $L^2$ spaces,
which can be equipped with $C(\mathcal{M})$-module structures. (A
function $f\in C(\mathcal{M})$ acts in $L^2$-spaces on $M$ and
$\partial M$ as the multiplication by $f$ and $f|_{\partial M}$,
respectively.) Moreover, the operator commutes with the
$C(\mathcal{M})$-module structure modulo compact operators. (This
is clear for smooth functions and follows by continuity for
continuous functions.)

Thus, our operator is an abstract elliptic operator in the sense of
Atiyah on $\mathcal{M}$ \cite{Ati4}. Hence a standard construction
(e.g., see~\cite{NaSaSt3}) defines the corresponding class in
analytic $K$-homology of $\mathcal{M}$, which we denote by
\begin{equation}\label{klass}
[D_M\otimes 1_{\im \mathcal{P}}]\in K^0(C(\mathcal{M}))\equiv
K_0(\mathcal{M}).
\end{equation}

\section{Poincar\'e isomorphism}

\paragraph{1. Main theorem.}

\begin{theorem}\label{theo1}
The mapping
\begin{equation}
\label{poi1}
\begin{array}{ccc}
  \tau_\mathcal{M}\colon K_0(\mathcal{A})&\longrightarrow &K^0(C(\mathcal{M})),\vspace{3mm}\\
   \mathcal{P} &\longmapsto &[D_M\otimes 1_{\im\mathcal{P}}], \\
\end{array}
\end{equation}
is an isomorphism.
\end{theorem}

\begin{proof}

0. It is easy to verify that the mapping  $\mathcal{P} \mapsto
[D_M\otimes 1_{\im\mathcal{P}}]$ induces a well-defined
homomorphism of the $K$-group into the $K$-homology group. (Indeed,
zero projections are taken to trivial operators by
Lemma~\ref{lem2}, and homotopic projections are taken to homotopic
operators.)

1. On the one hand, the exact sequence
$$
\begin{array}{ccccccc}
  0\to &C(X,\mathcal{K}) &\longrightarrow &\overline{\mathcal{A}}
  &\longrightarrow&
C(M)&\to 0\\
  & F & \longmapsto & (0,F) &  &  &  \\
  &   &  & (f,F) & \longmapsto & f &  \\
\end{array}
$$
of $C^*$-algebras, where the ideal consists of compact Toeplitz
symbols and the quotient consists of continuous functions on $M$,
induces the exact sequence
\begin{equation}
\label{easy1} \ldots  K_1(C(M))\stackrel{\delta_1}\to
K_0(C(X,\mathcal{K}))\to K_0(\overline{\mathcal{A}})\to
K_0(C(M))\stackrel{\delta_0}\to K_1(C(X,\mathcal{K}))\ldots
\end{equation}
in $K$-theory. By definition, the boundary mapping $\delta_0$ takes
the class of each vector bundle  $E$ on $M$ (the range of a
projection over $C(M)$) to the index
$$
\ind (D_\Omega\otimes 1_{E|_{\partial M}})\in K^1(X)\simeq
K_1(C(X,\mathcal{K}))
$$
of the family of Dirac operators in the fibers of $\pi$ twisted by
the restriction of  $E$ to $\partial M$.

2.  On the other hand, the short exact sequence
$$
0\to C_0(\mathcal{M}\setminus X)\longrightarrow
C(\mathcal{M})\longrightarrow C(X)\to 0
$$
of function algebras induces the sequence
\begin{equation}
\label{easy2} \dotsm  K^1(C_0(\mathcal{M}\setminus X))\to
K^0(C(X))\to K^0(C(\mathcal{M}))\to K^0(C_0(\mathcal{M}\setminus
X))\to K^1(C(X))\dotsm
\end{equation}
in  $K$-homology.

3. We combine the sequences \eqref{easy1} and \eqref{easy2} in the
diagram
\begin{equation}
\label{biga1}
\begin{array}{ccccccccc}
K_1(C(M))
  & \to & K_0(C(X,\mathcal{K})) & \to & K_0(\overline{\mathcal{A}}) & \to &
    K_0(C(M)) & \to & \!\!\!\!K_1(C(X,\mathcal{K}))\!\!\!\! \\
  \tau_M\downarrow &  & \downarrow\tau_X &  & \downarrow\tau_{\mathcal{M}} &  & \downarrow\tau_M &  & \downarrow\tau_X  \\
  \!\! K^1(C_0(\mathcal{M}\!\!\setminus\!\! X))\!\! & \to & K^0(C(X)) & \to & K^0(C(\mathcal{M})) &
  \to & \!\! K^0(C_0(\mathcal{M}\!\!\setminus\!\! X))\!\! & \to & K^1(C(X)) \\
\end{array}
\end{equation}
Here $\tau_X$ and $\tau_{M}$ are the Poincar\'e isomorphisms on $X$
and $M$. (Poincar\'e isomorphisms on smooth manifolds and manifolds
with boundary are considered, e.g., in~\cite{Kas2,MePi2}.) More
precisely, the mappings of the even $K$-groups are defined in terms
of Dirac operators twisted by corresponding vector bundles. The
definition of $\tau_X$ and $\tau_{M}$ on the odd $K$-groups can be
obtained by a suspension argument.

4. Let us show that the diagram \eqref{biga1} commutes neglecting
sign.

First, let us prove that the square
\begin{equation}
\label{sq1}
\begin{array}{ccc}
K_0(C(M))
  & \stackrel{\delta_0}\to & K_1(C(X,\mathcal{K}))  \\
  \tau_M\downarrow &  & \downarrow\tau_X   \\
  K^0(C_0(\mathcal{M}\setminus X)) & \to & K^1(C(X)) \\
\end{array}
\end{equation}
commutes. Indeed, let $E$ be a vector bundle on $M$; then the
composition of mappings passing through the top right corner
in~\eqref{sq1} takes the class $[E]\in K_0(C(M))$ to
$$
[D_X]\cdot \ind (D_\Omega\otimes 1_{E|_{\partial M}})\in K^1(C(X)),
$$
i.e., to the class of the Dirac operator on $X$ twisted by $\ind
(D_\Omega\otimes 1_{E|_{\partial M}})\in K_1(C(X))$. If we apply
the composition of mappings passing through the bottom left corner
of the square to $[E]$, then we obtain the element
$$
\pi_* [D_{\partial M}\otimes 1_{E|_{\partial M}}]\in K^1(C(X)).
$$
\begin{lemma}
One has
\begin{equation}
\label{tensor1} \pi_* [D_{\partial M}\otimes 1_{E|_{\partial
M}}]=[D_X]\cdot \ind (D_\Omega\otimes 1_{E|_{\partial M}}) \in
K^1(C(X)).
\end{equation}
\end{lemma}

\begin{proof}[Proof \rm (see~\cite{NaSaSt3})]
Consider the zero-order pseudodifferential operator
$\mathcal{D}=\Delta_{\partial M}^{-1/2}\circ(D_{\partial M}\otimes
1_{E|_{\partial M}})$ on the total space  $\partial M$ of $\pi$ as
a pseudodifferential operator on $X$ with operator-valued symbol in
the sense of~\cite{Luk1}. Then by the generalized Luke formula
\cite{NaSaSt3} we obtain
$$
\pi_*[\mathcal{D}]=Q[\ind \sigma_L(\mathcal{D})]\in K^1(C(X)),
$$
where $\sigma_L(\mathcal{D})$ is a symbol of compact fiber
variation on $T^*X$, $\ind \sigma_L(\mathcal{D}) \in
K_1(C_0(T^*X))$ is its index, and
$$
Q\colon K_1(C_0(T^*X))\longrightarrow K^1(C(X))
$$
is the mapping, induced by quantization, that takes the class of
each elliptic symbol to the class of the corresponding operator in
$K$-homology.

The symbol of $\mathcal{D}$ in the sense of \cite{Luk1} is (up to
the invertible factor $(\xi^2+\Delta_\Omega)^{1/2}$ defined by the
Laplacian $\Delta_{\partial M}$) the exterior tensor product
$$
\sigma(D_X)(x,\xi)\#(D_{\Omega_x}\otimes 1_{E|_{\Omega_x}}).
$$
Hence we obtain
$$
\ind \sigma_L(\mathcal{D})=[\sigma(D_X)]\ind(D_{\Omega}\otimes
1_{E|_{\partial M}})
$$
by the multiplicative property of the index. An application of $Q$
gives the desired relation \eqref{tensor1}.

The proof of the lemma is complete.
\end{proof}

Thus, the square~\eqref{sq1} commutes. By the suspension argument,
one proves that the same is true of the leftmost square
in~\eqref{biga1}.

Let us prove the commutativity of the square
\begin{equation}
\label{sq2}
\begin{array}{ccc}
K_0(C(X))
  & \to & K_0(\overline{\mathcal{A}})  \\
  \tau_X\downarrow &  & \downarrow\tau_\mathcal{M}   \\
  K^0(C(X)) & \to & K^0(C(\mathcal{M})). \\
\end{array}
\end{equation}
Let $[P]\in K_0(C(X))$ be the class of a vector bundle $\im P$ on
$X$. The composition of mappings passing through the top left
corner of~\eqref{sq2} takes $[P]$ to the $K$-homology class of the
operator
$$
D_X\otimes 1_{\im P}.
$$
On the other hand,  the composition of mappings in the opposite
order takes this element (by Lemma~\ref{lem2}) to the $K$-homology
class of the operator
$$
D_X^*\otimes 1_{\im P}.
$$
It is clear that these two classes are equal neglecting sign.

The square
\begin{equation}
\label{sq3}
\begin{array}{ccc}
K_0(\overline{\mathcal{A}})
  & \to & K_0(C(M))  \\
  \tau_{\mathcal{M}}\downarrow &  & \downarrow\tau_{M}   \\
  K^0(C(\mathcal{M})) & \to & K^0(C_0(\mathcal{M}\setminus X)) \\
\end{array}
\end{equation}
commutes for obvious reasons. (The horizontal mappings are
forgetful mappings.)

Thus, diagram \eqref{biga1} commutes (neglecting sign).

5. By applying the five lemma to diagram \eqref{biga1}, we complete
the proof of Theorem~\ref{theo1}.

\end{proof}

\paragraph{2. Obstruction to Fredholm realizations of Dirac operators
on manifolds with edges.} Let $D_M$ be a Dirac operator on $M$
associated with some $\spin^c$-structure on $M$ that induces
$\spin^c$-structures on the base and fibers of $\pi\colon \partial
M\longrightarrow X$. Let $D_X$ and $D_\Omega$ be the Dirac operator
on the base and the family of Dirac operators on the fibers,
respectively.

The operator $D_M$ defines a class
$$
[D_M]\in K^{\dim M}(C_0(\mathcal{M}\setminus X))
$$
in $K$-homology. Consider the inclusion $j\colon
C_0(\mathcal{M}\setminus X)\to C(\mathcal{M})$.
\begin{proposition}
\label{obst1} The element  $[D_M]\in K^{\dim
M}(C_0(\mathcal{M}\setminus X))$ is in the range of the mapping
$$
j^*\colon K^{*}(C(\mathcal{M}))\longrightarrow
K^*(C_0(\mathcal{M}\setminus X))
$$
of analytic $K$-homology groups if and only if Assumption
~\ref{triva} is satisfied, i.e., if and only if
$$
\ind D_\Omega=0.
$$
\end{proposition}

\begin{proof}

The $K$-homology exact sequence induced by $j$ has the form
\begin{equation}
\label{easy2a} \dotsm  K^{*+1}(C_0(\mathcal{M}\setminus X))\to
K^*(C(X))\to K^*(C(\mathcal{M}))\stackrel{j^*}\to
K^*(C_0(\mathcal{M}\setminus X))\stackrel{\partial }\to
K^{*+1}(C(X))\dotsm
\end{equation}
By exactness, we have  $[D_M]=j^*x$ for some $x\in
K^*(C(\mathcal{M}))$ if and only if  $\partial [D_M]=0$. The
element $\partial [D_M]$ is equal to
$$
\partial [D_M]=[D_X]\cdot \ind D_\Omega\in K^*(C(X))
$$
(see \eqref{tensor1}). Since $[D_X]$ is a generator of $K^*(C(X))$
as a free $K_*(C(X))$-module, we conclude that $\partial [D_M]=0$
is equivalent to $\ind D_\Omega=0$.

The proof of proposition is complete.
\end{proof}

\end{document}